\documentclass[10pt]{article}

\usepackage{amssymb,amsfonts,amsmath}
\usepackage[all,arc]{xy}
\usepackage{enumerate}
\usepackage{mathrsfs}
\usepackage{tikz}
\usepackage{tikz-cd}
\usepackage[]{algorithm2e}
\usetikzlibrary{decorations.markings,arrows,automata,arrows,backgrounds,snakes}
\usepackage{comment}
    \usetikzlibrary{decorations.markings,arrows,automata}

\input xy
\xyoption{all}

\newtheorem{thm}{Theorem}[section]

\newtheorem{rem}[thm]   {Remark}
\newtheorem{defn}[thm]  {Definition}

\newcounter{foo}  \Alph{foo}

\newenvironment{proof}  {\par\noindent{\bf Proof}\ }
                             {\hfill$\Box$\par\medskip}

\newenvironment{example}
{\medskip\par\noindent{\sc Example}\ }
{\par}

\begin{document}

\tikzset{->-/.style={decoration={
  markings,
  mark=at position .5 with {{->}}},postaction={decorate}}}

\title{Cosheaf Theoretical Constructions in Networks and Persistent Homology}

\author{Nicholas A. Scoville, Karthik Yegnesh}


\maketitle

\begin{abstract}
{Persistent homology has recently emerged as a powerful technique in topological data analysis for analyzing the emergence and disappearance of topological features throughout a filtered space, shown via persistence diagrams. Additionally, (co)sheaves have proven to be powerful instruments in tracking locally defined data across global systems, resulting in innovative applications to network science. In this paper, we combine the topological results of persistent homology and the quantitative data tracking capabilities of cosheaf theory to develop novel techniques in network data flow analysis. Specifically, we use cosheaf theory to construct persistent homology in a framework geared towards assessing data flow stability in hierarchical recurrent networks (HRNs). We use cosheaves to link topological information about a filtered network encoded in persistence diagrams with data associated locally to the network. From this construction, we use the homology of cosheaves as a framework to study the notion of ``persistent data flow errors." That is, we generalize aspects of persistent homology to analyze the lifetime of local data flow malfunctions. 
We study an algorithmic construction of persistence diagrams parameterizing network data flow errors, thus enabling novel applications of statistical methods to study data flow malfunctions. We conclude with an application to network packet delivery systems. }
{hierarchical recurrent networks, cosheaves, data flow, persistent homology, network packet delivery systems}
\\
2000 Math Subject Classification: 05C82, 54B40, 18F20
\end{abstract}

\section{Introduction}

Applications of (co)sheaves and sheaf cohomology to network science have recently emerged in the burgeoning field of topological data analysis from the works of Ghrist \cite{GH-2011} and Curry \cite{C-13}. The motivation is that sheaves provide a link between local and global data associated to portions of a topological space, enabling one to extract useful information pertaining to information flow from sheaves on a network. A fruitful application of (co)sheaf theory is sheaf cohomology, or dually cosheaf homology. This has previously been used in network analysis with an information-theoretic incentive \cite{GH-2011} to study information flow via network coding sheaves.\\
	
Additionally, persistent homology has surfaced as a powerful computational and algebraic tool in topological data analysis to study the emergence and disappearance of topological features across a filtered topological space. Its applications and uses are wide ranging. See for example \cite{ABDGL-15, B-15, DFW} and the recent book \cite{0-2015}. A particularly useful construction in persistent homology is that of a persistence diagram, which provides a visual depiction of the lifetime of topological features in the extended real half plane.\\
	
In this paper, we consider an approach to analyzing data flow in hierarchical recurrent networks (HRNs) that combines both the topological approach of persistent homology and the quantitative information tracking abilities of cosheaves. Neural networks carrying the topological structure of a HRN have a variety of applications ranging from speech recognition \cite{SB-2009} to image classification \cite{ZSW-16}. Our main results generalize techniques in persistent homology to study the notion of "persistent" data flow malfunctions in HRNs.\\

The paper is organized as follows. Sections \ref{Preliminaries} and \ref{Data Flow Control and Errors} provide an introduction to HRNs, our main object of study which will be used to model data flow in a network. Section \ref{Cosheaf Theoretical Constructions} gives a cosheaf theoretical construction of the notions in the previous sections and develops an application of homology to study data flow errors, as well as an algorithmic construction of generalized persistence based on results from Section \ref{Data Flow Control and Errors}. We conclude in Section \ref{Complex Networks Application} with a discussion of applications to complex networks.

\section{Preliminaries}\label{Preliminaries}

In this section, we establish the terminology and notation used throughout the body of the paper.


\subsection{Hierarchical Recurrent Networks}\label{Hierarchical Recurrent Networks}

Let $\mathbb{N}=\{0,1,2,3,\ldots\}$. We will use standard graph theory language.  Our reference for the basics of graph theory is \cite{West-96}

By a directed cycle, we mean a cycle on $k\geq 3$ vertices $v_1, \ldots, v_k$ with directed edge given by $v_i\to v_{i+1}$ for all $1\leq i \leq k-1$ and $v_k\to v_1$.

\begin{defn}\label{loop}  Let $G$ be a directed path on $2m+1$ vertexes $v_1\ldots, v_{2m+1}$ with directed edge $v_{2i} \to v_{2i-1}$ and $v_{2i}\to v_{2i+1}$ for each $1\leq i\leq m$.  If $H^1, \ldots, H^m$ are any $m$ directed cycles, a \textit{Hierarchical Recurrent Network} (HRN) $H$ is constructed by identifying a directed edge in cycle $H^i$ with edge $v_{2i}\to v_{2i-1}$ such that the orientations are consistent. Each $H^i$ is called the \textit{$i^{th}$ recurrent subprogram of $H$}, or sometimes just subprogram.

\end{defn}

\begin{example}\label{HRN example} As an example to illustrate Definition \ref{loop}, the directed graph $H$ below is a HRN with $m=3$.

$$
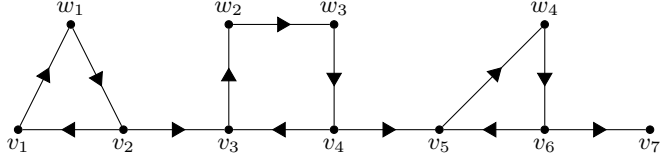
\begin{tikzpicture}[
    decoration={markings,mark=at position 0.6 with {\arrow{triangle 60}}},
    scale=1.4]

\node[inner sep=1pt, circle, fill=black] (1) at (0,0) [draw] {};
\node[inner sep=1pt, circle, fill=black] (2) at (1,0) [draw] {};
\node[inner sep=1pt, circle, fill=black] (3) at (2,0) [draw] {};
\node[inner sep=1pt, circle, fill=black] (4) at (3,0) [draw] {};
\node[inner sep=1pt, circle, fill=black] (5) at (4,0) [draw] {};
\node[inner sep=1pt, circle, fill=black] (6) at (5,0) [draw] {};
\node[inner sep=1pt, circle, fill=black] (7) at (6,0) [draw] {};

\node[inner sep=1pt, circle, fill=black] (9) at (0.5,1) [draw] {};
\node[inner sep=1pt, circle, fill=black] (10) at (2,1) [draw] {};
\node[inner sep=1pt, circle, fill=black] (11) at (3,1) [draw] {};
\node[inner sep=1pt, circle, fill=black] (12) at (5,1) [draw] {};

\draw[->-]  (2)--(1);
\draw[->-]  (2)--(3);
\draw[->-]  (4)--(3);
\draw[->-]  (4)--(5);
\draw[->-]  (6)--(5);
\draw[->-]  (6)--(7);

\draw[->-]  (9)--(2);
\draw[->-]  (1)--(9);

\draw[->-]  (10)--(11);
\draw[->-]  (3)--(10);
\draw[->-]  (11)--(4);

\draw[->-]  (12)--(6);
\draw[->-]  (5)--(12);

\node[anchor = north ]  at (1) {\small{$v_1$}};
\node[anchor = north ]  at (2) {\small{$v_2$}};
\node[anchor = north ]  at (3) {\small{$v_3$}};
\node[anchor = north ]  at (4) {\small{$v_4$}};
\node[anchor = north ]  at (5) {\small{$v_5$}};
\node[anchor = north ]  at (6) {\small{$v_6$}};
\node[anchor = north ]  at (7) {\small{$v_7$}};
\node[anchor = south ]  at (9) {\small{$w_1$}};
\node[anchor = south ]  at (10) {\small{$w_2$}};
\node[anchor = south ]  at (11) {\small{$w_3$}};
\node[anchor = south ]  at (12) {\small{$w_4$}};

\end{tikzpicture}
$$

The $i^{th}$ subprogram, $1\leq i\leq 3$, is given by
\begin{eqnarray*}
H^1&:& v_1\to w_1 \to v_2 \to v_1\\
H^2&:&v_3\to w_2 \to w_3 \to v_4 \to v_3\\
H^3&:&v_5 \to w_4 \to v_6\to v_5.\\
\end{eqnarray*}

\end{example}

\subsection{Representation of HRN Data Flow}\label{Representation of HRN Data Flow}

Although the bare structure of an HRN can be conveniently encoded as a directed graph, we need more structure to capture the notion of \textit{data flow} inside an HRN as well as error detection. In order to accomplish this, we will make use of a \textit{representation} of an HRN; i.e associating vector spaces to each node of an HRN and linear maps to the directed edges to represent the magnitudes of data that they contain.

\begin{defn}Let $H$ be a HRN, and  $H^n$ be the $n^{th}$ recurrent subprogram of $H$.  Suppose $H^n$ has vertex set $\{v_1,v_2,...,v_k\}$ and direction $\{e_i:v_i\rightarrow v_{i+1}\}$ and $e_k\colon v_k\to v_1$, and for each $n$, let $\sigma_i^n=\sigma_i$ be nonnegative integers, $1\leq i\leq k$. To each $v_i\in\ H^n$, we assign a vector space $\lambda_i^n$ over a field $k$  such that $\dim(\lambda_i^n)\leq\sigma_i$. The value $\sigma_i$ is the \textit{capacity} of vertex $v_i$.

Additionally, each directed edge $e_i$ is assigned a linear map $\phi_i\colon \lambda_i^n\to \lambda_{i+1}^n$, with $\phi_i$ a $\dim(\lambda_i^n)\times (\dim(\lambda_{i+1}^n)+\ell_n)$ matrix for a fixed $\ell_n\in\mathbb{N}$. This constitutes a \textit{quiver representation} $\lambda^n$ of $H^{n}$. When there is no chance of confusion, we refer to the quiver representation of a subprogram as a subprogram.

\end{defn}

\begin{example}\label{data flow example}
Let $H$ be as in Example \ref{HRN example}. We choose capacities $2,3,2$ for $H^1$, capacities $3, 5, 5, 0$ for $H^2$ and capacities $3,3,3$ for $H^3$ along with $\ell_1=2,\ell_2=0$, and $\ell_3=1$. Letting $\lambda(i)$ denote a vector space of dimension $i$, a quiver representation of $H^1, H^2$, and $H^3$ is given by

$$
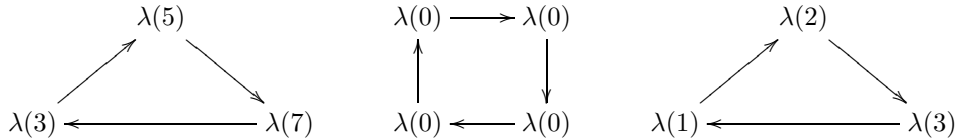
\xymatrix{
&\lambda(5) \ar[rd] & &                   \lambda(0) \ar[r] & \lambda(0) \ar[d] & & \lambda(2) \ar[rd]\\
\lambda(3) \ar[ru] && \lambda(7) \ar[ll] &  \lambda(0) \ar[u]& \lambda(0) \ar[l]& \lambda(1) \ar[ru] & & \lambda(3)\ar[ll]\\
}
$$

\noindent respectively.  The idea here is that data (in the form of vector spaces) is cycled through $H^n$, and each linear map ``adds" data to its domain vector space (i.e increases its dimension by a value of $\ell_n$). Notice that the single occurrence of $0$ in the capacities for $H^2$ forces all vector spaces to be $0$-dimensional as well as $m_2=0$.
\end{example}

We now wish to allow the data to cycle through $H^{n}$ multiple times. We can accomplish this by enhancing the quiver representation $\lambda^{n}$ by first defining a \textit{singular data flow} and using it to more generally define the concept of \textit{data flow}.

\begin{defn}\label{defn of singular data flow}
Let $H^{n}$ be a subprogram on $k$ vertices in an HRN $H$ and let $\lambda^{n}$ be its associated quiver representation. A \textit{singular data flow} is the sequence of vector spaces $\lambda^{n}_{1},\lambda^{n}_{2} ,\ldots,\lambda^{n}_{k}$ of $\lambda^n$.
\end{defn}

Note that Definition \ref{defn of singular data flow} does not include the linear transformations between vector spaces since a singular data flow represents the travel of data through the vertices in $H^{n}$.

\smallskip

We can now define a general data flow, i.e a data flow built from multiple singular data flows.

\begin{defn}
Let $H^{n}$ be a subprogram on $k$ vertices in an HRN $H$ and let $\lambda^{n}$ be its associated quiver representation, $m_n=m$ a positive integer.  A \textit{data flow} $D^n_m$ over  $\lambda^{n}$ is a sequence of vector spaces $\lambda_{1,1}^n, \lambda_{2,1}^n, \ldots, \lambda_{k,1}^n, \lambda_{1,2}^n, \lambda_{2,2}^n,\ldots, \lambda_{m,2}^n,\ldots,\lambda^n_{m,k}$. Note that many data flows can be associated to a single subprogram $\lambda^n$.
\end{defn}

\begin{rem}
Since index $m$ is the total number of times that the data flow circulates through the cycle $H^{n}$, it is clear that a singular data flow is the same as a data flow which iterates through the cycle exactly once.
\end{rem}

A data flow $D^n_m$ in $\lambda^{n}$ can be described by the following diagram of vector spaces:\\

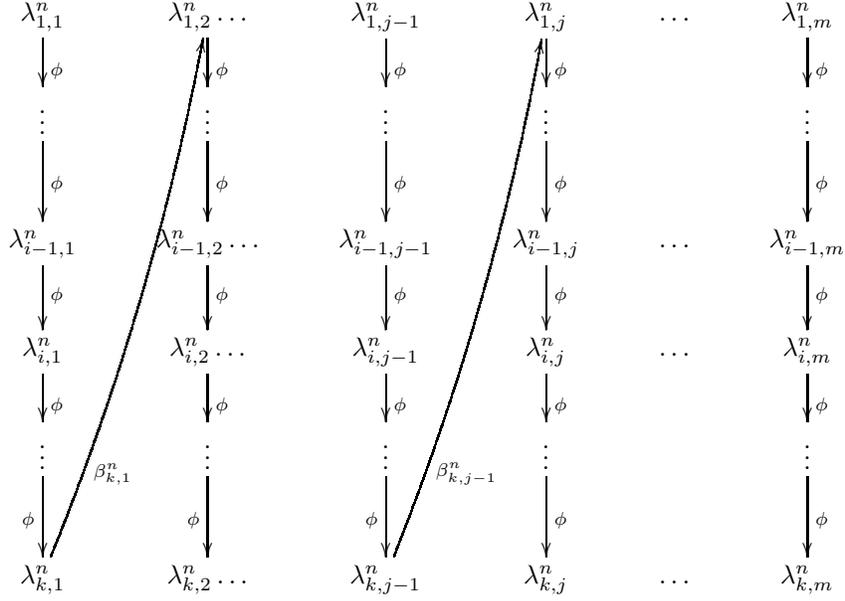
\begin{displaymath}
\xymatrix{
\lambda_{1,1}^n  \ar[d]^{\phi}& \lambda^n_{1,2} \ar[d]^{\phi}\ldots   & \lambda_{1,j-1}^n  \ar[d]^{\phi}  & \lambda_{1,j}^n  \ar[d]^{\phi}  & \ldots  &\lambda_{1,m}^n \ar[d]^{\phi}\\
\vdots \ar[d]^{\phi} & \vdots \ar[d]^{\phi} &\vdots \ar[d]^{\phi} & \vdots \ar[d]^{\phi} &  & \vdots \ar[d]^{\phi}\\
  \lambda_{i-1,1}^n  \ar[d]^{\phi} & \lambda^n_{i-1,2} \ar[d]^{\phi}\ldots & \lambda_{i-1,j-1}^n  \ar[d]^{\phi}  & \lambda_{i-1,j}^n \ar[d]^{\phi}  & \ldots  &\lambda_{i-1,m}^n \ar[d]^{\phi}\\
 \lambda_{i,1}^n \ar[d]^{\phi} & \lambda^n_{i,2} \ar[d]^{\phi}\ldots  & \lambda_{i,j-1}^n   \ar[d]^{\phi} & \lambda_{i,j}^n \ar[d]^{\phi}& \ldots  &\lambda_{i,m}^n \ar[d]^{\phi}\\
 \vdots \ar[d]_{\phi} & \vdots \ar[d]_{\phi} &\vdots \ar[d]_{\phi} & \vdots \ar[d]^{\phi} &  & \vdots \ar[d]^{\phi}\\
 \lambda_{k,1}^n  \ar@/_.5pc/[ruuuuu]_(.2){\beta^n_{k,1}}& \lambda^n_{k,2} \ldots & \lambda_{k,j-1}^n   \ar@/_.5pc/[ruuuuu]_(.2){\beta^n_{k,j-1}}& \lambda_{k,j}^n   & \ldots  & \lambda_{k,m}^n
 }
 \end{displaymath}

The sequence of vector spaces in each column represents a singular data flow and the linear map $\beta_{k,q}^n\colon \lambda_{k,q}^n\to \lambda_{1,q+1}^n$ ``connects" each singular data flow together that makes up the data flow.\\

Although it is not of great relevance to the following sections, we may consider a representation $\lambda$ of the \textit{entire} HRN $H$ once we are given representations $\lambda^{n}$ of all the cycles $H^{n}$. Note that there is exactly one directed edge between the ``last" vertex in cycle $H^{i}$ and the ``first" vertex in cycle $H^{i+1}$. We will impose that no data addition occurs from the representation $\lambda^{i}$ restricted to its last vertex and of $\lambda^{i+1}$ restricted to its first vertex, i.e the two vector spaces are equal (this makes sense because we think of the linear maps in $\lambda^{n}$ as ``adding data" by mapping its domain to a vector space of potentially higher dimension). We then require that the representation of the edges that are not part of a cycle are identity maps.

\section{Data Flow Control and Errors}\label{Data Flow Control and Errors}

For the following definitions, fix a data flow $D^n_m$ for each subprogram representation $\lambda^{n}$.

\begin{defn}
To each $\lambda^n$, assign a natural number $\delta^n$, the \textit{desired data output}. For any data flow $D^n_m$ of $\lambda^n$, the \textit{final data dimension of $\lambda^n$ with respect to $D^n_m$} is $\theta^n:=\dim(\lambda^{n}_{k,m})$. We sometimes say ``final data output of $\lambda^n$" with $D^n_m$ understood.
\end{defn}

The desired data output is the ideal output of information that we would like to accumulate within the subprogram while the final data dimension is the maximum possible amount of information that any particular data flow can output.  We are thus interested in whether or not any particular data flow has the capacity to accumulate the ideal amount of information for the subprogram.

\begin{defn}
A subprogram $\lambda^n$ is \textit{faulty with respect to $D^n_m$} (or just faulty) if $\theta^n<\delta^n$ while a subprogram is \textit{able with respect to $D^n_m$} (able) if $\theta^n>\delta^n$. If $\lambda^n$ is faulty, define $\widetilde{\lambda^{n}}:=k^{\delta^n-\theta^n}$ where $k$ denotes the ground field. Similarly, if $\lambda^n$ is able, define $\widehat{\lambda^{n}}:=k^{\theta^n-\delta^n}$.  We will occasionally replace the phrase ``$\lambda^n$ is faulty" with ``a data flow error occurs in $\lambda^n$" and replace the phrase ``$\lambda^n$ is able" with ``a data flow fix occurs in $\lambda^n$." For completeness, we call a subprogram \textit{sufficient with respect to $D^n_m$} (sufficient) if $\theta^n=\delta^n$.
\end{defn}

The point of these definitions is that data flow deficits created by faulty subprograms can be compensated for based on data surpluses produced by subsequent able subprograms. We will use the tools of cosheaf theory to express this mathematically and to create a framework in which one can systematically analyze these data flow errors as ``fixes," in the same way one would study the life and death of holes in a space using persistent homology.

\begin{example} From the three subprograms $\lambda^1, \lambda^2$, and $\lambda^3$ in Example \ref{data flow example}, we obtain three singular data flows $D^1_1,D^2_1$, and $D^3_1$ with $\theta^1=7$, $\theta^2=0$, and $\theta^3=3$.  If we choose desired data outputs $\delta^1=1, \delta^2=2$, and $\delta^3=3$, then we see that $\lambda^1$ is able, $\lambda^2$ is faulty, and $\lambda^3$ is sufficient. Note that $\lambda^2$ can never be able.
\end{example}

Note that one can determine if a singular data flow is faulty or able if one knows the initial value $\dim(\lambda^n_1), \ell_n$, and the number of vertices $k$.



 \section{Main Results}\label{Cosheaf Theoretical Constructions}
 We first recall some relevant information about cosheaves. This is a very brief exposition; we refer the reader to \cite{C-13, B-68, B-97, 0-2015}  for a more in-depth introduction to cosheaf theory and homology.
 \subsection{Cosheaves}\label{Cosheaves}

\begin{defn}Let $X$ be a topological space and $J$ an abelian category.
A \textit{$J$-valued precosheaf} $F$ on $X$ is a covariant functor $F\colon \mathrm{Open}(X) \to J$ from the category of open subsets of $X$ and inclusion maps to $J$. If $U\subset X$, an element $x\in F(U)$ is a \textit{cosection} of $F$ over $U$. For a pair of embedded open subsets $V \subset U \subset X$, the induced map on the inclusion $F(V)\to F(U)$ is called the \textit{corestriction map.} A precosheaf $F$ on $X$ is a \textit{cosheaf} if for any open $U \subset X$ and any open cover $\{U_i\}$ of $U$, the following sequence is exact:
$\bigoplus_iF(\bigcap_iU_i)  \to \bigoplus_iF(U_i) \to F(U) \to 0.$
\end{defn}

In this section, our domain space for cosheaves will be the real line $\mathbb{R}$ endowed with the standard topology. The subprograms $\lambda^n$ will be represented by points $n\in \mathbb{Z} \subseteq \in\mathbb{R}$. We will construct precosheaves $G_{1}, G_{2}:\mathrm{Open}(\mathbb{R})\rightarrow\mathrm{Vect}_{k}$ as follows. For $U\subset V$, let $p_{U}$ denote the largest $n\in U$ that is the index of a subprogram. If $\lambda^n$ is faulty, then $G_{1}(U)=\widetilde{\lambda^{p_{U}}}$. If $\lambda^n$ is able, then $G_{1}(U)=0$. Hence $G_{1}$ records the margins of error in the maximal point of $U$. We now define $G_{2}$, which records the margins of surplus data of the maximal point of $U$. If $\lambda^n$ is able, then $G_{1}(U)=\widehat{\lambda^{p_{U}}}$. If $\lambda^n$ is faulty, then $G_{1}(U)=0$.
\begin{defn} Let $X$ be a topological space and $\mathscr{U}=\{U_i\}$ an open cover of $X$, and $F$ a precosheaf of abelian groups on $X$. The group of \textit{$\check{C}$ech $k$-chains} associated to $\mathscr{U}$ is the group $C_k(\mathscr{U},F)= \bigoplus_iF(U_{0,1,\ldots ,k})$, where $U_{0,1,\ldots,k}=\bigcap_{i=0}^kU_i$.

Equipped with differentials $\partial_k \colon C_k(\mathscr{U};F)\to C_{k-1}(\mathscr{U};F)$, we obtain a \textit{$\check{C}$ech complex} $C_*(\mathscr{U};F)= 0 \to C_k(\mathscr{U};F)\to C_{k-1}(\mathscr{U};F)\to \ldots \to C_0(\mathscr{U};F)\to 0$. We denote the $n^{th}$ $\check{C}$ech homology group associated to $F$ and covering $\mathscr{U}$ by $\check{H}_n(\mathscr{U};F)$.
\end{defn}

\subsection{Main Results}

In this section, we state our results generalizing notions in the persistent homology theory of filtered topological spaces to studying network data flow security.
	
We present a theorem allowing for a homological description of the emergence of data flow errors and resolutions. This allows us to construct \textit{error P-intervals}, which parametrize the lifetime of network data flow malfunctions.
We will make use of a specfic covering of $\mathbb{R}$, which we now describe:\\

Let $\mathscr{U}=\{U_i\}$ be a (co)filtered open covering of $\mathbb{R}$ with $U_m\subset U_{m-1}\subset...\subset U_0=\mathbb{R}$ and $k\in U_r$ if and only if $r\leq k$. From now on, $\mathscr{U}$ will denote this specific covering throughout the paper.\\

Let $T[n]$ be the shift endofunctor over $\mathrm{Ch}(\mathrm{Vect}_{k})$ (shifts indexes of a chain complex up $n$). Also, let $\beta_{>k}$ denote the brutal truncation endofunctor on $\mathrm{Ch}(\mathrm{Vect}_{k})$, trivializing all terms indexed $j>k$. Let $\varphi^{k}_{n}:\beta_{>k}$  $\circ\beta_{<k}(\check{C}_{k}(\mathscr{U};G_n))\rightarrow\beta_{>k}\circ\beta_{<k}\circ T[1](\check{C}_{k-1}(\mathscr{U};G_n))$ be the projection.

\begin{thm}\label{det}
A data flow error (resp. fix) occurs in the $k^{th}$ subprogram of $\lambda^{*}$ if and only if $H_{k}(\mathrm{Cone}(\varphi^{k}_{n}))$ is non trivial for $n=1$ (resp. $n=2$).
\end{thm}

\begin{proof}
Let ${\check{C}_{k}(\mathscr{U};G_n)}^{!}$ and ${\check{C}_{k-1}(\mathscr{U};G_n)}^{!}$ denote the domain and codomain of $\varphi^{k}_{n}$, respectively.
     Clearly, $H_{k}(\mathrm{Cone}(\varphi^{k}_{n}))\simeq ker(\varphi^{k}_{n})$. Since every exact sequence in $\mathrm{Vect}_k$ splits, the exact sequence: $0\to \ker(\varphi^{k}_{n})\to {\check{C}_{k}(\mathscr{U};G_n)}^{!}\xrightarrow{\varphi^{k}_{n}}{\check{C}_{k-1}(\mathscr{U};G_n)}^{!}\rightarrow 0$ splits and we obtain that $ker(\varphi^{k}_{n})\simeq {\check{C}_{k}(\mathscr{U};G_n)}^{!}/{\check{C}_{k-1}(\mathscr{U};G_n)}^{!}$. From the cofiltered nature of $\mathscr{U}$, it holds that ${\check{C}_{k}(\mathscr{U};G_n)}^{!}/{\check{C}_{k-1}(\mathscr{U};G_n)}^{!}$ yields the amount of data flow deficit (resp. surplus) created at subprogram $k$ for $n=1$ (resp. $n=2$). Thus, a non trivial $H_{k}(\mathrm{Cone}(\varphi^{k}_{n}))$ reveals the emergence of a network data flow error or resolution for $n=1$ or $n=2$, respectively. The converse is trivial.
\end{proof}

 \subsection{``Persistent Homology" for Data Flow Errors}
 In this section, we will provide motivation for the subsequent subsections by providing basic definitions regarding persistent homology. The next subsections will use the results of the previous subsection (a homological description of data flow errors and fixes) to formulate an application of persistent homology to study the birth and death of data flow errors (a data flow fix is seen as the death of an error).\\

 Basically, persistent homology is a tool which one uses to study the birth and death of topological features in a filtered space. Formally:
\begin{defn}
Let $X_{\bullet}$ be a filtered topological space, i.e a space $X$ equipped with a sequence of nested subspaces $X_{0}\subset X_{1}\subset\ldots\subset X_{n}=X$. Let $H_n(X)$ denote the $n^{th}$ singular homology (with coefficients in a field $k$) vector space of $X$. Fix indexes $i$ and $j$ with $i\leq j$. The $(i,j)$-persistent $n^{th}$ homology vector space $H^{i,j}_{n}(X_{\bullet})$ is defined as $H^{i,j}_{n}(X_{\bullet})=\mathrm{im}(H_n(X_i)\rightarrow H_n(X_j))$, where $H_n(X_i)\rightarrow H_n(X_j)$ is induced by the inclusion $X_i\subset X_j$.
\end{defn}
\begin{rem}
The dimension of the $k$-vector space $H^{i,j}_{n}(X_{\bullet})$ is the number of $n$-dimensional holes present in subspace $X_{i}$ that are also present in $X_{j}$.
\end{rem}
\begin{example}
If $i=j$, then it is clear that $H^{i,j}_{n}(X_{\bullet})=H_{n}(X_{i})$ since the $k$-linear map induced by $\mathrm{id}:X_{i}\rightarrow X_{i}$ must be the identity map on $H_{n}(X_{i})$.
\end{example}

\begin{defn}
Let $X_{\bullet}$ be a filtered topological space.
Let $\mathbb{Z}_{\infty}$ denote the set $\mathbb{Z}\cup\{\infty\}$. A degree $n$ persistence diagram is a multiset over $\mathbb{Z}_{\infty}\times\mathbb{Z}_{\infty}$, where the multiplicity $\mu(i,j)$ of a point $(i,j)\in\mathbb{Z}_{\infty}$ is $\mathrm{dim}(H_{n}^{i,j}(X_{\bullet}))$. The points parametrize the lifetimes of homology classes (``holes") in $X_{\bullet}$, i.e the presence of a point $(i,j)$ means that a hole was created at index $i$ in the filtration and destroyed at index $j$. The multiplicity function records how many holes were created and destroyed for each interval.
\end{defn}

\begin{defn}
 The intervals $(i,j)$ in the persistence diagram are called \textit{homological P-intervals}, or simply P-intervals.
\end{defn}
 \begin{defn}
  A point $(i,j)$ is called \textit{proper} if $j<\infty$ and a \textit{point at infinity} otherwise. A proper point indicates the presence of holes in $X_{\bullet}$ that are born and die within the ``bounds" of the filtration. An improper point indicates the presence of a hole that is created but never dies.
 \end{defn}

\subsection{Error $P$-intervals}\label{More Homological Constructions}
We can adapt the notion of a P-interval from persistent homology to the notion of an \textit{error P-interval}. Instead of parametrizing the lifetime of a hole in a filtered topological space, an error P-interval parametrizes the lifetime of a data flow error in an HRN.


We can now give an algorithmic description of \textit{$P$-intervals} which parametrize the lifetimes of network data flow errors. First, we fix some notation. Let $S^{1}$ denote the set of indexes of all the subprograms in $\lambda$ that are faulty. Let $S^{2}$ denote the set of indexes of all the subprograms in $\lambda$ that are able.\\

\definecolor{light-gray}{gray}{0.35}
\begin{algorithm}[H]
\KwData{$S^n$ for $n\in\{1,2\}$    \textcolor{light-gray}{//Lists of indexes of subprograms with nontrivial error and fix contributions for $n=1$ and $n=2$, respectively}}
 \KwResult{Set of \textit{$P$-intervals} $\mathscr{P}\subset\mathbb{Z}\times\{\mathbb{Z}\cup\infty\}$}
 \While{$i\leq \max(S^1)$}{$\{$
 Starting from $i=\min(S^1)$\\
\textcolor{light-gray}{//pairs index of subprogram in which network error is created (in $S^{1}$) with index in which error is resolved (in $S^{2}$)}\\
Assign to $i\in S^1$ the least $k\in S^2$ not already chosen such that $\dim(\ker(\varphi_1^{i}))=\dim(\ker(\varphi_2^{k}))$;\\
increment $i$ by one$\}$;\\
  \eIf{$|S^1|=|S^2|$}{
   $\mathscr{P}=\{(i,k)\}$, where $i$ is assigned to $k$\;
   }{
   $\mathscr{P}=\{(i,k)\}\cup\{(S^1\setminus S^2)\times\infty$\} \textcolor{light-gray}{//Accounts for points at infinity}\;
  }
 }
\end{algorithm}

Algorithm 1 assigns each error index from $S^1$ to its ``fixing" index in $S^2$ based on availability of space in \textit{able} subprograms. Thus, the output $\mathscr{P}$ can be thought of as a persistence diagram, since a point $(i,j)\in\mathscr{P}$ indicates that a data flow deficit occurs in subprogram $\lambda^i$ but $\lambda^j$ has enough data flow surplus to completely account for the deficit created.\\

The set of assignments constitute proper points, i.e they indicate the presence of faulty subprograms whose data flow deficit is accounted for by a subsequent surplus. Additionally, elements in $S^1$ that do not have ``fixing" indexes in $S^2$ are paired with $\infty$, constituting points at infinity.

\subsection{Persistent Error Homology Groups}

We can now provide a construction of persistent error homology groups as an analogue to the persistent homology groups of a filtered topological space. This is due to the observation that a persistence diagram $\mathrm{Dgm}(X)$ determines completely determines the persistent homology groups $H^{i,j}_{k}$ of a filtered topological space $X$, concretely stated by the following: $\dim(H^{i,j}_{k}(X))=\sum_{\{u,v: u\leq i, v\leq j\}}\mu(u,v)$, where $\mu:\mathrm{Dgm}(X)\ni(i,j)\rightarrow\mathbb{N}$ is the multiplicity function. We adjust this statement to define persistent error homology groups of an HRN given a persistence diagram generated by Algorithm1. We can define the \textit{(i,j) - persistent error homology group} $H^{i,j}(\mathscr{V}_{\lambda^{*}})$ of $\mathscr{V}_{\lambda^{*}}$ by
$\dim(H^{i,j}(\mathscr{V}_{\lambda^{*}}))=\sum_{\{u,v: u\leq i, v\leq j\}}\mu(u,v)$, where $(i,j)\in\widetilde{\mathrm{Dgm}}_{\mathscr{V}_{\lambda^{*}}}$ and $j<\infty$. In this context, the multiplicity $\mu(i,j)$ of a proper point $(i,j)$ records the dimension of data flow deficit created at index $i$ and resolved at index $j$.

\section{Complex Networks Application}\label{Complex Networks Application}
In this section, we state the applications of our results to complex networks, along with explicitly describing an application to Internet package delivery systems.\\

The most fruitful applications of this work stem from the results generalizing concepts in persistent homology to study problems in network science. Particularly, our technique to construct error persistence diagrams and providing a visual parametrization of the lifetime of data flow errors in HRNs creates a framework for the novel application of statistical tools to assess data flow malfunctions recorded in persistence diagrams. Now results from such an assessment would have implications not in the topological features of a network (as have been shown previously), but in the nature of its data flow. \\

From immediate inspection of an error persistence diagram, one can see the significance of a single data flow error to overall network security. This is based on the distance from the point to the diagonal of $\mathbb{Z}_{\infty}\times\mathbb{Z}_{\infty}$. An accumulation of points along the diagonal indicates data flow ``noise," as opposed to points far from the diagonal, which show the presence of significant data flow malfunctions within the network. More interestingly, extended metrics on the space of (error) persistence diagrams can be used to evaluate similarities between instances of data flow across the same network, as we demonstrate below.\\

Let $\mathscr{V}_{\lambda^{*}_1}$ and $\mathscr{V}_{\lambda^{*}_2}$ be two instances of data flow in an HRN represented in $\mathrm{Vect}_k$ by $\lambda^{*}_{1}$ and $\lambda^{*}_{2}$, respectively. By instances of data flow, we mean two different situations in which data flows through the same network. Practically, this situation emerges if the HRN in question acts as a model for a frequently occurring process. Let $\widetilde{\mathrm{Dgm}}_{\mathscr{V}_{\lambda^{*}_1}}$ and $\widetilde{\mathrm{Dgm}}_{\mathscr{V}_{\lambda^{*}_2}}$ be the error persistence diagrams of $\mathscr{V}_{\lambda^{*}_{1}}$ and $\mathscr{V}_{\lambda^{*}_{2}}$, respectively. Let $\varphi\colon \widetilde{\mathrm{Dgm}}_{\mathscr{V}_{\lambda^{*}_1}}\to \widetilde{\mathrm{Dgm}}_{\mathscr{V}_{\lambda^{*}_2}}$ be a bijection of persistence diagrams.
Recall that the \textit{bottleneck distance} $d_{B}(\widetilde{\mathrm{Dgm}}_{\mathscr{V}_{\lambda^{*}_1}},\widetilde{\mathrm{Dgm}}_{\mathscr{V}_{\lambda^{*}_2}})$ between persistence diagrams $\widetilde{\mathrm{Dgm}}_{\mathscr{V}_{\lambda^{*}_1}}$ and $\widetilde{\mathrm{Dgm}}_{\mathscr{V}_{\lambda^{*}_2}}$ is defined as $d_{B}(\widetilde{\mathrm{Dgm}}_{\mathscr{V}_{\lambda^{*}_1}},\widetilde{\mathrm{Dgm}}_{\mathscr{V}_{\lambda^{*}_2}})=\inf_{\varphi} \sup_{p} \|p-\varphi(p)\|_{\infty}$, where $p\in\widetilde{\mathrm{Dgm}}_{\mathscr{V}_{\lambda^{*}_1}}$. The bottleneck distance functions as an extended metric to compare similarities between persistence diagrams, with a lesser distance indicating greater similarities in the distribution of points throughout each diagram. In our network theoretic context, $d_{B}(\widetilde{\mathrm{Dgm}}_{\mathscr{V}_{\lambda^{*}_1}},\widetilde{\mathrm{Dgm}}_{\mathscr{V}_{\lambda^{*}_2}})$ is an indicator of the differences between the persistence of data flow errors between $\mathscr{V}_{\lambda^{*}_1}$ and $\mathscr{V}_{\lambda^{*}_2}$.\\

As a toy example of an application of these techniques, consider the following portion of the HRN $H$ defined in Example  \ref{HRN example} above which models a recurrent network packet delivery system (i.e a packet-switched network arranged in a HRN structure). The output of our techniques can be seen as a framework to statistically study error-detection and correction.
$$
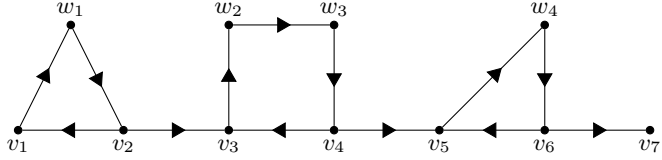
\begin{tikzpicture}[
    decoration={markings,mark=at position 0.6 with {\arrow{triangle 60}}},
    scale=1.4]

\node[inner sep=1pt, circle, fill=black] (1) at (0,0) [draw] {};
\node[inner sep=1pt, circle, fill=black] (2) at (1,0) [draw] {};
\node[inner sep=1pt, circle, fill=black] (3) at (2,0) [draw] {};
\node[inner sep=1pt, circle, fill=black] (4) at (3,0) [draw] {};
\node[inner sep=1pt, circle, fill=black] (5) at (4,0) [draw] {};
\node[inner sep=1pt, circle, fill=black] (6) at (5,0) [draw] {};
\node[inner sep=1pt, circle, fill=black] (7) at (6,0) [draw] {};

\node[inner sep=1pt, circle, fill=black] (9) at (0.5,1) [draw] {};
\node[inner sep=1pt, circle, fill=black] (10) at (2,1) [draw] {};
\node[inner sep=1pt, circle, fill=black] (11) at (3,1) [draw] {};
\node[inner sep=1pt, circle, fill=black] (12) at (5,1) [draw] {};


\draw[->-]  (2)--(1);
\draw[->-]  (2)--(3);
\draw[->-]  (4)--(3);
\draw[->-]  (4)--(5);
\draw[->-]  (6)--(5);
\draw[->-]  (6)--(7);

\draw[->-]  (9)--(2);
\draw[->-]  (1)--(9);

\draw[->-]  (10)--(11);
\draw[->-]  (3)--(10);
\draw[->-]  (11)--(4);

\draw[->-]  (12)--(6);
\draw[->-]  (5)--(12);

\node[anchor = north ]  at (1) {\small{$v_1$}};
\node[anchor = north ]  at (2) {\small{$v_2$}};
\node[anchor = north ]  at (3) {\small{$v_3$}};
\node[anchor = north ]  at (4) {\small{$v_4$}};
\node[anchor = north ]  at (5) {\small{$v_5$}};
\node[anchor = north ]  at (6) {\small{$v_6$}};
\node[anchor = north ]  at (7) {\small{$v_7$}};
\node[anchor = south ]  at (9) {\small{$w_1$}};
\node[anchor = south ]  at (10) {\small{$w_2$}};
\node[anchor = south ]  at (11) {\small{$w_3$}};
\node[anchor = south ]  at (12) {\small{$w_4$}};

\end{tikzpicture}
$$

The nodes of $H$ represent the location of delivery sites and the directed edges model the flow of network packet data from one node to another. We will illustrate our procedure on this system. Since this network is quite small, it does not fully illustrate the capabilities of our techniques (which are most powerful in large HRNs), but describing the procedure explicitly and visually serves as a demonstrative tool. The reader can assume that $H$ is the first part of a larger HRN which continues on to the right.\\

Denote by $H^1$, $H^2$, and $H^3$ the three subprograms in $H$. We will refer to subsequent subprograms (not shown) to the extending to the right of the picture by $H^{i}$ for $i\geq 4$. Let $\lambda^{n}$ be their respective vector space representations and fix a data flow $D^n$ in each $\lambda^n$. In this case, one can choose the vector spaces $\lambda^{n}_{p,q}$ in each $D^{n}$ to represent the magnitude of the network packets in its position, i.e $\mathrm{dim}(\lambda^{n}_{p,q})$ is a natural number indicating the quantity of formatted data units contained within its vertex. We assign desired data outputs $\delta^n$ to each $\lambda^n$. Denote the final data output of each $\lambda^n$ (as determined by $D^n$) by $\theta^n$. Let $F$ denote the set of indexes of the faulty subprograms, i.e $\delta^i>\theta^i$ for all $i\in F$, ordered by increasing magnitude. Let $A$ denote the set of indexes of all able subprograms, i.e $\delta^i<\theta^i$, ordered in the same way. We represent each $\lambda^n$ by a point $n\in\mathbb{R}$, where $\mathbb{R}$ is endowed with the standard topology. In the manner of section \ref{Cosheaves}, define precosheaves $G_1, G_2:\mathrm{Open}(\mathbb{R})\rightarrow\mathrm{Vect}_{k}$. We recall the definition on open sets (see section \ref{Cosheaves} for notation and more explanation): If $\lambda^n$ is faulty, then $G_{1}(U)=\widetilde{\lambda^{p_{U}}}$. If $\lambda^n$ is able, then $G_{1}(U)=0$. $G_{1}$ records the margins of error in the maximal point of $U$. If $\lambda^n$ is able, then $G_{1}(U)=\widehat{\lambda^{p_{U}}}$. If $\lambda^n$ is faulty, then $G_{1}(U)=0$. By Theorem \ref{det}, we can detect the margin of deficit (resp. surplus) contributed by a faulty (resp. able) subprogram $\lambda^i$ algebraically based on the dimension of the vector space $H_{k}(\mathrm{Cone}(\varphi^{k}_{n}))\simeq \ker(\varphi^{k}_{n})$. Thus we can input the ordered sets $F$ and $A$ as $S^1$ and $S^2$ in Algorithm 1 ``Generation of $P$-intervals" to obtain a set of points $\{(i,j)\in\mathbb{Z}_{\infty}\times\mathbb{Z}_{\infty}\}$. A point $(i,j)$ parametrizes the lifetime of a single network packet malfunction, i.e a error is created at subprogram $i$ and is detected and fully corrected by subprogram $j$. Outputted points $(k,\infty)$ are points at infinity and therefore are errors in $H$ that are never detected and corrected. The multiplicity of any point $i,j)$ is precisely the dimension of the data flow error created at $i$ and resolved at $j$.\\

The outputted persistence diagram whose points parametrize the lifetime and magnitude of data flow malfunctions across the packet delivery system allows the application of statistical tools to assess the clustering of points, which in this specific case reveals information regarding the clustering of packet delivery malfunctions across various time slots. This in turn provides novel information to enable systematic assessments of flaws in the network packet delivery system and pathways to fixing them.

\subsection*{Acknowledgments}  The authors would like to thank Dr. Chad Giusti and Dr. Robert Ghrist for helpful discussions and ideas, as well as an anonymous referee who greatly helped improve the readability of the paper.

\bibliographystyle{amsplain}
\bibliography{Cosheaf}

\end{document}